\documentclass[12pt]{amsart}
\usepackage{amssymb}
\usepackage{mathtools}
\usepackage{txfonts}
\usepackage{eucal}
\usepackage{extarrows}

\usepackage{color}

\usepackage[all]{xy}

\usepackage{tikz}

\usepackage{graphicx}
\usepackage{float}

\usepackage{xspace}

\usepackage[a4paper,body={16.3cm,21.8cm},centering]{geometry}

\usepackage[colorlinks,final,backref=page,hyperindex,hypertex]{hyperref}

\newcommand{\nc}{\newcommand}
\newcommand{\delete}[1]{}

\nc{\mlabel}[1]{\label{#1}}  
\nc{\mcite}[1]{\cite{#1}}  
\nc{\mref}[1]{\ref{#1}}  
\nc{\mbibitem}[1]{\bibitem{#1}} 

\delete{
\nc{\mlabel}[1]{\label{#1}  
{\hfill \hspace{1cm}{\small\tt{{\ }\hfill(#1)}}}}
\nc{\mcite}[1]{\cite{#1}{\small{\tt{{\ }(#1)}}}}  
\nc{\mref}[1]{\ref{#1}{{\tt{{\ }(#1)}}}}  
\nc{\mbibitem}[1]{\bibitem[\bf #1]{#1}} 
}

\newtheorem{theorem}{Theorem}[section]
\newtheorem{prop}[theorem]{Proposition}
\newtheorem{lemma}[theorem]{Lemma}
\newtheorem{coro}[theorem]{Corollary}

\theoremstyle{definition}
\newtheorem{defn}[theorem]{Definition}
\newtheorem{remark}[theorem]{Remark}

\newtheorem{prob}[theorem]{Problem}
\newtheorem{prop-def}{Proposition-Definition}[section]

\newcommand\alphlist{a,b,c,d,e,f,g,h,i,j,k,l,m,n,o,p,q,r,s,t,u,v,w,x,y,z}
\newcommand\Alphlist{A,B,C,D,E,F,G,H,I,J,K,L,M,N,O,P,Q,R,S,T,U,V,W,X,Y,Z}
\newcommand\getcmds[3]{\expandafter\newcommand\csname #2#1\endcsname{#3{#1}}}
\makeatletter
\@for\x:=\alphlist\do{\expandafter\getcmds\expandafter{\x}{frak}{\mathfrak}}
\@for\x:=\Alphlist\do{\expandafter\getcmds\expandafter{\x}{frak}{\mathfrak}}
\makeatother

\nc{\bfk}{{\bf k}}
\font\cyr=wncyr10

\newfont{\scyr}{wncyr10 scaled 550}
\nc{\sha}{\mbox{\cyr X}}
\nc{\ssha}{\mbox{\bf \scyr X}}

\nc{\Id}{\mathrm{Id}}
\nc{\lbar}[1]{\overline{#1}}

\nc{\leaf}{\mathrm{leaf}}

\usetikzlibrary{calc}

\newcommand \lie{\mathfrak{g}} 
\newcommand \plie{P_{\lie}}
\newcommand \rblie{(\mathfrak{g}, P_{\mathfrak{g}})} 
\newcommand \tril{\triangleleft} 
\newcommand \trir{\triangleright} 
\newcommand \gv{\lie \natural V} 
\newcommand \np{\mathcal{P}} 
\newcommand \pe{P_E}

\nc{\li}[1]{\textcolor{red}{#1}}
\nc{\lir}[1]{\textcolor{red}{Li:#1}}
\nc{\yi}[1]{\textcolor{blue}{Yi: #1}}
\nc{\xing}[1]{\textcolor{purple}{Xing:#1}}
\nc{\revise}[1]{\textcolor{red}{#1}}

\begin{document}

\title[Extending structures of Rota-Baxter Lie algebras]{Extending structures of Rota-Baxter Lie algebras}
\author{Xiao-song Peng}
\address{School of Mathematics and Statistics,
Jiangsu Normal University, Xuzhou, Jiangsu 221116, P.\,R. China}
\email{pengxiaosong3@163.com}

\author{Yi Zhang$^{*}$}\footnotetext{${*}$Corresponding author.}
\address{School of Mathematics and Statistics,
	Nanjing University of Information Science \& Technology, Nanjing, Jiangsu 210044, P.\,R. China}
\email{zhangy2016@nuist.edu.cn}


\date{\today}
\begin{abstract}
In this paper,  we first introduce the notion of an extending datum of a Rota-Baxter Lie algebra  through a vector space. We then construct a unified product for the Rota-Baxter Lie algebra  with a vector space as a main ingredient in our approach. Finally, we solve the extending structures problem of Rota-Baxter Lie algebras, which generalizes and unifies two  problems in the study of Rota-Baxter Lie algebras: the extension problem studied by Mishra-Das-Hazra and the factorization problem investigated by Lang-Sheng.
\end{abstract}

\subjclass[2010]{17B38,  17B05, 17B60
}

\keywords{}

\maketitle

\tableofcontents

\setcounter{section}{0}

\allowdisplaybreaks

\section{Introduction}
The aim of this paper is to solve the extending structure problem for Rota-Baxter Lie algebras which asks for the classification of all Rota-Baxter Lie algebraic structures on a given vector space by the approach of a general unified product.

\subsection{Rota-Baxter (Lie) algebras}
Let $\mathbf{k}$ be a field and $\lambda \in \mathbf{k}$. A {\bf Rota-Baxter algebra } of weight $\lambda$ is a pair $(R,P)$, where $R$ is an associative algebra over $\mathbf{k}$, and $P:R\rightarrow R$ is a linear operator which satisfies the {\bf Rota-Baxter identity}
\begin{align}
P(x)P(y)=P(xP(y)+P(x)y+\lambda xy) \, \text{ for }\, x, y\in R. \label{eq:RBid}
\end{align}
Then $P$ is called a {\bf Rota-Baxter operator} of weight $\lambda$.
The notion of a Rota-Baxter algebra was probably introduced in the first time by G. Baxter~\cite{Bax60} in 1960 based on the study of Spitzer's identity in fluctuation theory.
It is a natural  algebraic interpretation of  the formula of partial integration in analysis, as in the case of a differential algebra could be treated as an algebraic abstraction of differential equations. More precisely, let $R$ be the $\mathbb{R}$-algebra of continuous functions on $\mathbb{R}$. Define $P: R\rightarrow R$ to be the integration
\begin{align*}
P(f)(x)=\int_{0}^{x}f(t)dt.
\end{align*}
Then the integration by parts formula
\begin{align*}
\int_{0}^{x}P(f)'(t)P(g)(t)dt=P(f)(x)P(g)(x)
-\int_{0}^{x}P(f)(t)P(g)'(t)dt
\end{align*}
is precise the Eq.~(\ref{eq:RBid}) with $\lambda=0$.

After 60 years developments, the study of Rota-Baxter
algebras has become a vibrant research field and   plays an important role in mathematics and mathematical physics, such as multiple zeta values~\cite{GZ10}, pre-Lie algebras~\cite{Agu00, Bai20}, Hopf algebras~\cite{Gon21, MLCW22, YGT19, ZGG16, ZGG22}, (antisymmetric) infinitesimal bialgebras~\cite{Bai1, SW23}, dendriform algebras~\cite{EG08, Foi21, ZGM20}, operads~\cite{PBG17}, ~$\mathcal{O}$-operators~\cite{BGN1}, quantum field theory~\cite{CK}, classical (associative) Yang-Baxter equations~\cite{BGN1, BGM21, BGGZ, Gub21}. For more details about Rota-Baxter algebras, see the monograph~\cite{Gub}.

Completely independent of the above developments, the Rota-Baxter operator in the context of Lie algebras has its own motivations and developing history. In fact, it is an operator form of the classical Yang-Baxter equation (CYBE), named after the  physicists C. N. Yang~\cite{Yang} and R. Baxter~\cite{BaR}. For a given Lie algebra $\lie$, the {\bf classical Yang-Baxter equation} (CYBE) was  defined by the following tensor form
\begin{equation}\notag
[r_{12},r_{13}]+[r_{12},r_{23}]+[r_{13},r_{23}]=0,
\label{eq:cybe}
\end{equation}
where $r\in \lie \otimes \lie $. In~\cite{STS}, Semonov-Tian-Shansky showed
that if there exists a nondegenerate symmetric invariant bilinear
form on  $\lie$,  then a skew-symmetric  solution $r$ of
the CYBE  can be equivalently expressed
as a linear operator $R:\lie \to \lie$ satisfying the following identity
\begin{equation}\label{eq:CYBE-RB}
[R(x), R(y)]=R([R(x),y])+R([x,R(y)]),\;\;\forall x,y\in
\lie,\end{equation}
which is precisely the Rota-Baxter relation (of weight zero) in Eq.~(\ref{eq:RBid}) in the context of Lie algebras. In this way,  Eq.~(\ref{eq:CYBE-RB}) is called an {\bf operator form} of the CYBE. We would like to emphasize that this result was generalized by  Kupershmidt~\mcite{Ku} by the concept of an $\mathcal{O}$-operator, which was also called a relative Rota-Baxter operator.

Quite recently the discovery of some new algebraic and geometric structures in mathematics and physics~\cite{GLS21} has led to a renewed interest in the study of Rota-Baxter Lie algebras and their connections with deformations and homotopy theory~\cite{LST21, LST23, S22}, cf.~\cite{S22} for recent surveys. One of the main motivation comes from the emergence of Rota-Baxter Lie groups. In ~\cite{GLS21}, Lang-Sheng-Guo showed that the differentiation of
a Rota-Baxter Lie group is a Rota-Baxter Lie algebra of weight 1, generalizing the fact that the differentiation of a  Lie group is a  Lie algebra. Let $G$ be a Lie group. A {\bf Rota-Baxter Lie group} is a Lie group $G$ together with a smooth map  $\mathcal{B}:G\rightarrow G$ on $G$ satisfying
\begin{align*}
\mathcal{B}(g_1)\mathcal{B}(g_2)=\mathcal{B}(g_1 \mathrm{Ad}_{\mathcal{B}(g_1)}g_2), \ \ \forall g_1, g_2\in G,
\end{align*}
where $\mathrm{Ad}_g, g\in G$ is an adjoint action. Then the following formula
\begin{align*}
[u, v]_\lie=\frac{d^2}{dt ds}\mid_{t,s=0}\mathrm{exp}^{tu}\mathrm{exp}^{sv}\mathrm{exp}^{-tu}, \ \ \forall u,v \in \lie
\end{align*}
captures the relationship between the Lie bracket $[,]_\lie$ and the Lie group multiplication.
Let $e$ be an identity of $G$ and $\lie=T_e G$  the Lie algebra of $G$. Then the tangent map $B=\mathcal{B}_{\star e}:\lie \rightarrow \lie$ of $\mathcal{B}$  at the identity $e$ is a Rota-Baxter operator of weight $1$ and so $(\lie,[,]_\lie, B )$ is a Rota-Baxter Lie algebra of weight 1.

Viewing Rota-Baxter Lie algebras in the framework of bialgebras, Lang and Sheng~\cite{LS23}  introduced the notion of a quadratic Rota-Baxter Lie algebra, which has a  one-one correspondence to the factorizable Lie bialgebras. This leads to the  introduction of  a Rota-Baxter Lie bialgebra~\cite{LS23}. The study of Rota-Baxter Lie bialgebras was partly motivated by the theory of Manin triple approach. Except for the construction  of a one-to-one correspondence between Manin triples of Rota-Baxter Lie algebras  and Rota-Baxter Lie bialgebras~\cite{LS23},
Bai-Guo-Liu-Ma~\cite{BGLM22} independently
established a general bialgebra structure on Rota-Baxter Lie algebras following the approach of Manin triples to Lie bialgebras and antisymmetric infinitesimal bialgebras. A cohomology theory of a relative Rota-Baxter Lie algebra with coefficients in a representation was studied in~\cite{JS21}.


\subsection{ Extending structures problem}

The extending structures(ES) problem  arose in the study of group theory by Agore and  Militaru~\cite{AM141} around 2014. It can be regarded as a uniform of two well-known problems in group theory-- the {\bf extension problem} of H\"{o}lder~\cite{H95} and the {\bf factorization problem} of Ore~\cite{Ore37}, which  served as baby models for the ES-problem approach~\cite{AM141}. Later, Agore and  Militaru argued in favor of the study of  ES-problems on several occasions. For example, the extending structures problem for Lie algebras and unital associative algebras are solved in ~\cite{AM14} and ~\cite{AM16} respectively. Apparently, the ES-problem becomes more popular when some new motivations were found, coming on one hand from interesting applications to conformal algebras~\cite{HS17} and on the other from close relation to the theory of 3-Lie algebras~\cite{Z22} and superalgebras~\cite{ZCY19}.

Considering the fact that the Rota-Baxter Lie algebra is a generalization of the Lie algebra, it is almost a natural question to consider the following extending structures problem.
\begin{prob}
Let $\rblie$ be a Rota-Baxter Lie algebra. Let $E$ be a vector space containing $\lie$ as a subspace. Up to an isomorphism of Rota-Baxter Lie algebras which stabilizes $\rblie$ , describe
and classify all the Rota-Baxter Lie algebraic structures $([-, -], \pe)$ on  $E$     such that $\rblie$ is a Rota-Baxter Lie subalgebra of $(E,[-, -],\pe)$.
\end{prob}
In order to solve the above  problem, we follow the steps of Agore and Militaru~\cite{AM14} on the construction of the Lie algebra case, we give the concept of an extending datum of a Rota-Baxter Lie algebra through a vector space and a Rota-Baxter Lie extending structure. Generalizing the results of  Agore and Militaru~\cite{AM14} of the Lie algebra case, the description part and the classification part of the extending structure problem of Rota-Baxter Lie algebras are answered by a unified product of Rota-Baxter Lie algebras and a relative cohomology group.
In particular, the non-abelian extension of Rota-Baxter Lie algebras in~\cite{MDH23} and the factorization problem of Rota-Baxter Lie algebras in~\cite{LS23} appear to be special cases of the Rota-Baxter Lie extending structures problem.

This paper is organized as follows. In Section~\ref{sec:a}, we give the concept of an extending datum of Rota-Baxter Lie algebras with vector spaces and a unified product from an extending datum. Then the description part of the extending structure problem of Rota-Baxter Lie algebras is answered in Theorem~\ref{thm:indu}. The classification part of the extending structure problem of Rota-Baxter Lie algebras is answered in Theorem~\ref{thm:main} using a relative cohomology group. In Section~\ref{sec:b}, we show that the cross products and bicrossed products of two Rota-Baxter Lie algebras are special cases of the unified product. In Section~\ref{sec:c}, we consider the flag extending structure of Rota-Baxter Lie algebras and give an answer to the calculation of the classifying object of the extending structure of the flag extending structure of Rota-Baxter Lie algebras.

{\bf Notation.} Throughout this paper, let {\bf k} be a field unless the contrary is specified, which will be the base ring of all modules, algebras, coalgebras, bialgebras, tensor products, as well as linear maps.

\section{Unified products for Rota-Baxter Lie algebras}\label{sec:a}

In this section, we first introduce an extending datum of a Rota-Baxter Lie algebra. We then give a construction of the unified product of a Rota-Baxter Lie algebra with a vector space. We prove that any Rota-Baxter Lie algebra on a vector space $E$ containing a given Rota-Baxter Lie algebra as a Rota-Baxter Lie subalgebra is isomorphic to a unified product. Finally, the answer to the classification part of the extending structure problem is given.

\begin{defn}
Let $\lambda$ be a fixed element of $\bfk$. A {\bf Rota-Baxter Lie algebra of weight $\lambda$} is a Lie algebra $\lie$ together with a linear map $P_{\lie}$ such that
\begin{align*}
P[g,h]=P([P(g),h]+[g,P(h)]+\lambda [g,h]), \,\, \forall g,h \in \lie.
\end{align*}
Denote it by $(\lie, [\cdot, \cdot], P_\lie)$ or $\rblie$.
\end{defn}

{\em Throughout the rest of this paper}, we say $\rblie$ is a Rota-Baxter Lie algebra to mean that $\rblie$ is a Rota-Baxter Lie algebra of weight $\lambda$.

\begin{defn}
Let $\rblie$ and $(\lie',P_{\lie'})$ be two Rota-Baxter Lie algebras. A {\bf morphism
of Rota-Baxter Lie algebras} $f:\rblie \rightarrow (\lie',P_{\lie'})$ is a morphism of Lie algebras $f: \lie \rightarrow \lie'$ such that $f \circ P_{\lie}=P_{\lie'} \circ f$. It is said to be an {\bf isomorphism} if $f$ is a linear isomorphism.
\end{defn}

Let $E$ be a vector space and $\lie \subset E$  a subspace of $E$.
A subspace $V$ of $E$ is a {\bf complement} of $\lie$ in $E$ if $E=V+ \lie$ and $V \cap \lie=0$. Note that such a complement is unique up to isomorphism and its dimension is called the {\bf codimension} of $\lie$ in $E$.

For later use, we recall the notion of a Rota-Baxter module over a Rota-Baxter Lie algebra given in ~\cite{LS23}.

\begin{defn}
Let $\lie$ be a Lie algebra. A {\bf left $\lie$-module} is a vector space $V$ together with a linear map $\trir: \lie \otimes V \rightarrow V$, called the left action of $\lie$ on $V$, such that
\begin{align*}
[g, h] \trir x= g \trir (h \trir x)-h \trir (g \trir x),
\end{align*}
for all $x \in V$ and $g, h \in \lie$.
\end{defn}
The notion of right $\lie$-modules can be defined similarly. Note that any right $\lie$-module is a left $\lie$-module via $g \trir x :=-x \tril g$ and viceversa. Hence the category of right $\lie$-modules is isomorphic to the category of left $\lie$-modules.

\begin{defn}\cite[Definition~3.6]{LS23}
Let $\rblie$ be a Rota-Baxter Lie algebra. A {\bf left $\rblie$-module} is a pair $(V,T)$, where $V$ is a vector space and $T: V \rightarrow V$ is a linear map, together with a linear map $\trir : \lie \otimes V \rightarrow V$ such that $V$ is a left $\lie$-module under the map $\trir$ and for $g\in \lie$ and $v \in V$,
\begin{align*}
P_{\lie}(g) \trir T(v)=T(P_{\lie}(g) \trir v+g \trir T(v)+ \lambda g \trir v).
\end{align*}
\end{defn}
The notion of right $\rblie$-modules can be defined similarly.

\begin{defn}
Let $\rblie$ be a Rota-Baxter Lie algebra and let $(E,P_E)$ be a pair where $E$ is a vector space and $P_E:E \rightarrow E$ is a linear map. Suppose $\lie$ is a subspace of $E$ and $V$ is a complement of $\lie$ in $E$. For a linear map $\phi : E \rightarrow E$ such that $\phi \circ P_E=P_E \circ \phi$, we consider the diagram:
\begin{align*}
\xymatrix{
  \rblie \ar[d]_{id_{\lie}} \ar[r]^{i} & (E,P_E) \ar[d]_{\phi} \ar[r]^{\pi} & V \ar[d]^{id_V} \\
  \rblie \ar[r]^{i} & (E,P_E) \ar[r]^{\pi} & V   }
\end{align*}
where $\pi : E \rightarrow V$ is the canonical projection of $E = \lie + V$ on $V$ and $i: \lie \rightarrow E$ is
the inclusion map such that $i \circ P_{\lie}=P_E \circ i$. We say that $\phi: E \rightarrow E$ {\bf stabilizes} $\rblie$ (resp. {\bf co-stabilizes} $V$ ) if the left square (resp. the right square) of the above diagram is commutative.
\end{defn}

\begin{defn}
Suppose that $(E, [\cdot, \cdot], P)$ and $(E, [\cdot, \cdot]', P')$ are two Rota-Baxter Lie algebras on $E$ both containing $\rblie$ as a Rota-Baxter Lie subalgebra.
\begin{enumerate}
\item  $(E, [\cdot, \cdot], P)$ and $(E, [\cdot, \cdot]', P')$ are called {\bf equivalent},  denoted  by $(E, [\cdot, \cdot], P) \equiv (E, [\cdot, \cdot]', P')$, if there exists a Rota-Baxter Lie algebra isomorphism $\phi: (E, [\cdot, \cdot], P) \rightarrow (E, [\cdot, \cdot]', P')$ which stabilizes $\rblie$.
\item $(E, [\cdot, \cdot], P)$ and $(E, [\cdot, \cdot]', P')$ are called {\bf cohomologous}, denoted  by $(E, [\cdot, \cdot], P) \approx (E, [\cdot, \cdot]', \\P')$, if there exists a Rota-Baxter Lie algebra isomorphism $\phi: (E, [\cdot, \cdot], P) \rightarrow (E, [\cdot, \cdot]', P')$ which stabilizes $\rblie$ and co-stabilizes $V$.
\end{enumerate}
\label{def:eeq}
\end{defn}

\begin{remark}
\begin{enumerate}
\item Note that  $\equiv$ and $\approx$ are equivalence relations on the set of all Rota-Baxter Lie algebras structures on $E$ containing $\rblie$ as a Rota-Baxter Lie subalgebra and we denote by $Extd (E,\rblie)$ (resp. $Extd'(E,\rblie)$) the set of all equivalence classes via $\equiv$ (resp. $\approx$).

\item  By Definition~\ref{def:eeq}, $Extd(E,\rblie)$ gives the set of all isomorphism classes of Rota-Baxter Lie algebras on $E$ that stabilizes $\rblie$ and $Extd'(E,\rblie)$ gives the set of all isomorphism classes of Rota-Baxter Lie algebras on $E$ from the point of extension theory.
\item Since any two cohomologous Rota-Baxter Lie algebras structures on $E$ are equivalent, there is a canonical projection
\begin{align*}
Extd'(E,\rblie) \rightarrow Extd(E,\rblie).
\end{align*}
\end{enumerate}
\end{remark}

Generalizing the extending datum of Lie algebras through vector spaces~\cite{AM14}, we give the following notion.
\begin{defn}
Let $\rblie$ be a Rota-Baxter Lie algebra and let $V$ be a vector space.
\begin{enumerate}
\item An {\bf extending datum} of $\rblie$ through $V$ is a system $$\Omega(\rblie, V)=(\tril, \trir,f, \{-,- \}, P_1, P_2)$$ consisting of four bilinear maps
\begin{align*}
\tril: V \times \lie \rightarrow V, \, \trir: V \times \lie \rightarrow \lie, \, f: V \times V \rightarrow \lie, \, \{-,-\}: V \times V \rightarrow V
\end{align*}
and two linear maps
\begin{align*}
P_1: V \rightarrow \lie, \, P_2: V \rightarrow V.
\end{align*}

\item Let $\Omega(\rblie, V)=(\tril, \trir, f, \{-, -\}, P_1, P_2)$ be an extending datum of $\rblie$ through $V$. Denote by $\gv$ the vector space $\lie \times V$ with a bilinear map $[-, -] : (\lie \times V)\times (\lie \times V) \rightarrow \lie \times V$ defined by
\begin{align}
[(g, x), (h, y)]:= ([g, h]+x \trir h-y \trir g+ f(x, y), \{x, y\}+x \tril h-y \tril g)
\label{eq:uniproduct}
\end{align}
and a linear map $\np: \lie \times V \rightarrow \lie \times V$ defined by
\begin{align}
\np((g,x)) :=(\plie(g)+P_1(x), P_2(x)),
\label{eq:unilie}
\end{align}
for all $g, h \in \lie$ and $x, y \in V$. The object $\gv$ is called a {\bf unified product} of $\rblie$ and $V$ if it is a Rota-Baxter Lie algebra with the maps given by Eqs.~(\ref{eq:uniproduct}) and (\ref{eq:unilie}).
\item The extending datum $\Omega(\rblie, V)=(\tril, \trir, f, \{-, -\ \},P_1, P_2)$ is called a {\bf Rota-Baxter Lie extending structure} of $\rblie$ through $V$ if $\gv$ is a unified product of $\rblie$ and $V$.
\item Denote by $\mathcal{L}(\rblie, V)$ the set of all Rota-Baxter Lie extending structures of $\rblie$ through $V$.
\end{enumerate}
\end{defn}
Now we provide an equivalent characterization of the unified product $\gv$ of $\rblie$ and $V$ as follows.

\begin{theorem}
Let $\rblie$ be a Rota-Baxter Lie algebra. Let $V$ be a vector space and $\Omega(\rblie, V)$ an extending datum of $\rblie$ through $V$. Then the following statements are equivalent:
\begin{enumerate}
\item $\gv$ is a unified product of $\rblie$ and $V$;

\item For any $g, h \in \lie$ and $x,y, z \in V$,
\begin{enumerate}
\item \label{it:1} $f(x,x)=0, \, \{x,x\}=0$;

\item \label{it:2} $(V,\tril)$ is a right $\lie$-module;

\item \label{it:3} $x \trir [g,h]=[x \trir g, h]+[g, x \trir h]+(x \tril g) \trir h-(x \tril h) \trir g$;

\item \label{it:4} $\{x,y \} \tril g=\{x, y \tril g \}+\{ x \tril g, y\}+ x \tril(y \trir g)-y \tril (x \trir g)$;

\item \label{it:5} $\{x,y \} \trir g=x \trir (y \trir g)- y\trir (x \trir g)+[g,f(x,y)]+f(x,y \tril g)+f(x \tril g,y)$;

\item \label{it:6} $f(x, \{y,z \})+f(y,\{z,x \})+f(z, \{x,y \})+x \trir f(y,z)+ y \trir f(z,x)+ z \trir f(x,y)=0$;

\item \label{it:7} $\{x, \{y,z \} \}+\{y, \{ z, x\} \}+\{ z, \{ x,y\} \}+x \tril f(y,z)+ y \tril f(z,x)+ z \tril f(x,y)=0$;

\item \label{it:8} $[\plie(g), P_1(y)]-P_2(y) \trir \plie(g)+\plie(y \trir \plie(g))
+P_1(y \tril \plie(g))\\
- \plie([\lie, P_1(y)])
    + \plie(P_2(y) \trir g)
    +P_1(P_2(y) \tril g)+ \lambda \plie(y \trir g)\\
    + \lambda P_1(y \tril g)
    =0$;

\item \label{it:9} $P_2(y) \tril \plie(g)-P_2(y \tril \plie(g))-P_2(P_2(y) \tril g)- \lambda P_2(y \tril g)=0$;

\item \label{it:10} $[P_1(x),P_1(y)]+P_2(x) \trir P_1(y)-P_2(y) \trir P_1(x)+f(P_2(x),P_2(y))\\
    +\plie(y \trir P_1(x))
    -\plie(f(P_2(x),y))
    -P_1(\{P_2(x),y \})+P_1(y \tril P_1(x))\\
    - \plie(x \trir P_1(y))- \plie(f(x,P_2(y)))
    -P_1(\{ x, P_2(y)\})
    -P_1(x \tril P_1(y))\\
    - \lambda \plie(f(x,y))-\lambda P_1 (\{x,y \})
    =0$;

\item \label{it:11} $\{P_2(x),P_2(y) \}+P_2(x) \tril P_1(y)-P_2(y) \tril P_1(x)-P_2(\{P_2(x),y \})\\
    +P_2(y \tril P_1(x))-P_2(\{x,P_2(y) \})-P_2(x \tril P_1(y))-\lambda P_2(\{ x,y\})
    =0$.

\end{enumerate}
\end{enumerate}
\label{thm:cha}
\end{theorem}

\begin{proof}
By Theorem 3.2 of~\cite{AM14}, $\gv$ is a Lie algebra under the product given by~(\ref{eq:uniproduct}) if and only if Eqs.~\ref{it:1}-\ref{it:7} hold. Next we show that the operator $\np$ given by Eq.~(\ref{eq:unilie}) is a Rota-Baxter operator if and only if Eqs.~\ref{it:8}-\ref{it:11} hold. Since $(g, x)=(g,0)+(0, x)$ in $\gv$ and $\np$ given by Eq.~(\ref{eq:unilie}) is a linear map, we only need to consider the action of $\np$ on all generators of $\gv$, i.e. the set $\{(g, 0) | g \in \lie \} \cup \{(0, x) | x \in V \}$. As $[-,-]$ is antisymmetric, there are three steps to consider. First, $\np$ is a Rota-Baxter operator for any $(g,0),(h,0)$:
\begin{align*}
&\ [\np((g,0)), \np((h,0))]-\np\Big([\np(g,0), (h,0)]+[(g,0), \np(h,0)]+ \lambda [(g,0), (h,0)]\Big)\\
=&\ [(\plie(g),0), (\plie(h),0)]-\np\Big([(\plie(g),0),(h,0)]+[(g,0), (\plie(h),0)]+ \lambda[(g,0),(h,0)]\Big)\\
=&\ ([\plie(g), \plie(h)],0)-\Big( \plie([\plie(g),h]+[g, \plie(h)]+ \lambda [g,h]),0\Big) \quad(\text{by Eqs.~(\ref{eq:uniproduct})-(\ref{eq:unilie})})\\
=&\ 0  \quad  \text{(As $\plie$ being a Rota-Baxter operator of $\lie$)}.
\end{align*}
Next we prove that $\np$ is a Rota-Baxter operator for any $(g,0),(0,y)$ if and only if Eqs.~(\ref{it:8})-(\ref{it:9}) hold. Indeed,
\begin{align*}
&\ [\np((g,0)), \np((0,y))]-\np\Big([\np(g,0), (0,y)]+[(g,0), \np(0,y)]+ \lambda [(g,0), (0,y)]\Big)\\
=&\ [(\plie(g),0),(P_1(y), P_2(y))]-\np \Big([(\plie(g),0),(0,y)]+[(g,0),(P_1(y),P_2(y))]+ \lambda [(g,0), (0,y)] \Big)\\
& \hspace{11cm}(\text{by Eq.~(\ref{eq:unilie})})\\
=&\ \bigg([\plie(g), P_1(y)]- P_2(y) \trir \plie(y), -P_2(y) \tril \plie(g) \bigg)-\np \bigg(\Big(-y \trir \plie(g), -y \tril \plie(g)\Big)\\
&\ + \Big([g, P_1(y)]-P_2(y) \trir g, -P_2(y) \tril g\Big)+ \lambda \Big(-y \trir g,-y \tril g\Big) \bigg)\quad(\text{by Eq.~(\ref{eq:uniproduct})})\\
=&\ \bigg([\plie(g), P_1(y)]- P_2(y) \trir \plie(y), -P_2(y) \tril \plie(g) \bigg)\\
&\ -\bigg(\plie(-y \trir \plie(g))+P_1(-y \tril \plie(g)),P_2(-y \tril \plie(g)) \bigg)\\
&\ - \bigg(\plie([g, P_1(y)])-\plie(P_2(y) \trir g)+P_1(-P_2(y) \tril g), P_2(-P_2(y) \tril g) \bigg)\\
&\ - \lambda \Big(\plie(-y \trir g)+P_1(-y \tril g), P_2(-y \tril g) \Big)\quad(\text{by Eq.~(\ref{eq:unilie})})\\
=&\ \bigg([\plie(g),P_1(y)]-P_2(y) \trir \plie(g)+ \plie(y \trir \plie (g))+P_1( y \tril \plie(g))- \plie([g, P_1(y)])\\
&\ + \plie (P_2(y) \trir g)
 +P_1(P_2(y) \tril g)+ \lambda \plie(y \trir g)+ \lambda P_1(y \tril g),\\
 &\ P_2(y) \tril \plie(g)-P_2(y \tril \plie(g))-P_2(P_2(y) \tril g)
 -\lambda P_2(y \tril g) \bigg).
\end{align*}
Hence $\np$ is a Rota-Baxter operator for $(g,0), (0,y)$ if and only if Eqs.~\ref{it:8}-\ref{it:9} hold. Finally, we prove that $\np$ is a Rota-Baxter operator for $(0,x),(0,y)$ if and only if Eqs.~\ref{it:10}-\ref{it:11} hold. Indeed,
\begin{align*}
&\ [\np((0,x)), \np((0,y))]-\np \bigg([\np((0,x)),(0,y)]+[(0,x), \np((0,y))]+ \lambda [(0,x), (0,y)] \bigg)\\
=&\ \left[\Big(P_1(x),P_2(x)\Big), \Big(P_1(y), P_2(y)\Big)\right]-\np \bigg(\left[\Big(P_1(x),P_2(x)\Big), (0,y)\right]+\left[(0,x), \Big(P_1(y),P_2(y)\Big)\right]\\
&\ + \lambda [(0,x), (0,y)] \bigg)\\
=&\ \bigg([P_1(x),P_1(y)]+P_2(x) \trir P_1(y)-P_2(y) \trir P_1(x)+f\Big(P_2(x), P_2(y)\Big), \{P_2(x), P_2(y) \}\\
&\ +P_2(x) \tril P_1(y)
-P_2(y) \tril P_1(x) \bigg)-\np \bigg(\Big(-y \trir P_1(x)+f(P_2(x),y), \{P_2(x),y \}-y \tril P_1(x)\Big)\\
&\ +\Big(x \trir P_1(y)
 +f(x,P_2(y)), \{ x,P_2(y)\}+ x\tril P_1(y)\Big)+ \lambda \Big(f(x,y), \{ x,y\}\Big) \bigg)\\
=&\ \bigg([P_1(x),P_1(y)]+P_2(x) \trir P_1(y)-P_2(y) \trir P_1(x)+f(P_2(x), P_2(y)), \{P_2(x), P_2(y) \}\\
&\ +P_2(x) \tril P_1(y)
-P_2(y) \tril P_1(x) \bigg)-\bigg(\plie(-y \trir P_1(x))+ \plie(f(P_2(x),y))+P_1(\{P_2(x),y \})\\
&\ -P_1(y \tril P_1(x)), P_2(\{P_2(x),y \})
 -P_2(y \tril P_1(x)) \bigg)-\bigg(\plie(x \trir P_1(y))+ \plie(f(x, P_2(y)))\\
&\ +P_1(\{x, P_2(y) \})+P_1(x \tril P_1(y)),
 P_2(\{x, P_2(y) \}) + P_2(x \tril P_1(y)) \bigg)\\
&\ -\lambda \bigg(\plie(f(x,y))+P_1(\{ x,y\}), P_2(\{ x,y\}) \bigg)\\
=&\ \bigg([P_1(x),P_1(y)]+P_2(x) \trir P_1(y)-P_2(y) \trir P_1(x)+ f(P_2(x), P_2(y))+ \plie (y \trir P_1(x))\\
&\ - \plie (f(P_2(x), y))- P_1(\{P_2(x),y \})+P_1(y \tril P_1(x))- \plie(x \trir P_1(y))- \plie (f(x, P_2(y)))\\
&\ -P_1(\{ x,P_2(y)\})-P_1(x \tril P_1(y))- \lambda \plie (f(x,y))- \lambda P_1(\{ x,y\}), \{P_2(x),P_2(y) \}+ P_2(x) \tril P_1(y)\\
&\ -P_2(y) \tril P_1(x)-P_2(\{P_2(x), y \})+P_2(y \tril P_1(x))-P_2(\{x, P_2(y) \})-P_2(x \tril P_1(y))-\lambda P_2(\{x,y \}) \bigg).
\end{align*}
Hence $\np$ is a Rota-Baxter operator for $(0,x), (0,y)$ if and only if Eqs.~\ref{it:10} and \ref{it:11} hold. This completes the proof.
\end{proof}

\begin{remark}
Note that Eqs.~\ref{it:2} and \ref{it:11} show that $(V,P_2)$ is a right $\rblie$-module under the action $\tril$.
\end{remark}

Let $\Omega(\rblie, V)=(\tril, \trir, f,\{-, -\}, P_1, P_2)$ be a Rota-Baxter Lie extending structure and $\gv$ the associated unified product. Then there is a canonical inclusion
\begin{align*}
i_{\lie}: \lie \rightarrow  \gv , \quad i_{\lie}(g)=(g, 0),  \text{ \ for any \ } g\in \lie.
\end{align*}
Note that $i_{\lie} \circ P_\lie  =\np \circ i_\lie$.
Then $i_{\lie}$
is an injective morphism of Rota-Baxter Lie algebras. Hence $\rblie$ can be viewed as a Rota-Baxter Lie subalgebra of $\gv$  through the identification $\lie \cong i_{\lie}(\lie) \cong \lie \times \{0 \}$. Next we will prove that any Rota-Baxter Lie algebra structure on a vector space $E$ containing $\rblie$ as a Rota-Baxter Lie subalgebra is isomorphic to a unified product of $\rblie$ through a vector space.

Now we arrive at our first main result in this section.
\begin{theorem}
Let $\rblie$ be a Rota-Baxter Lie algebra. Let $E$ be a vector space containing $\lie$ as a subspace and $(E, [-,-], \pe)$ a Rota-Baxter Lie algebra such that $\rblie$ is a Rota-Baxter Lie subalgebra of $(E,[-, -],\pe)$. Then there exists a Rota-Baxter Lie extending structure $$\Omega(\rblie, V)= (\tril, \trir, f,\{-, -\}, P_1,P_2)$$ of $\rblie$ through a subspace $V$ of $E$ and an isomorphism of Rota-Baxter Lie algebras $$(E,\pe) \cong (\gv, \np)$$ that stabilizes $\rblie$ and co-stabilizes $V$.
\label{thm:indu}
\end{theorem}

\begin{proof}
Since $\bfk$ is a field, we have a linear map $p : E \rightarrow \lie$ such that $p(g)=g$, for all $g \in \lie$. Define $V:=\mathrm{ker}(p)$. Then $V$ is a complement of $\lie$ in $E$, $p \circ p=p$ and $p \circ (id_E-p)=0$. Now define the extending datum of $\rblie$ through $V$ as follows:
\begin{align*}
\trir:&\ V \times \lie \rightarrow \lie, \quad x \trir g:=p([x,g]), \\
\tril:&\  V \times \lie \rightarrow V, \quad x \tril g:=[x,g]-p([x,g]),\\
f:&\  V \times V \rightarrow \lie, \quad f(x,y):=p([x, y]),\\
\{,\}:&\ V \times V \rightarrow V, \quad  \{x, y\} := [x, y] - p([x, y]),\\
P_1:&\  V \rightarrow \lie, \quad P_1(x)=p(\pe(x)),\\
P_2:&\  V \rightarrow V, \quad P_2(x)=\pe(x)-p(\pe(x)),
\end{align*}
for any $g \in \lie$ and $x,y \in V$. Note that by definition of the map $p$, for any $g \in \lie$ and $x, y \in V$, $x \tril g \in V$, $\{x, y\} \in V$ and $P_2(x) \in V$. Hence the above maps are all well-defined linear maps or bilinear maps. Now we prove that $\Omega(\rblie, V)=(\tril, \trir, f, \{-, -\},P_1, P_2)$ is a Rota-Baxter Lie extending structure of $\rblie$ through $V$ and
\begin{align*}
\phi: \gv \rightarrow E,  \quad \phi(g, x):= g+x
\end{align*}
is an isomorphism of Rota-Baxter Lie algebras that stabilizes $\lie$ and co-stabilizes $V$. Analogue to the proof of Theorem~3.4 in~\cite{AM14},
$$\phi: \lie \times V \rightarrow E, \quad \phi(g, x) := g + x$$
is a linear isomorphism between the vector space $\lie \times V$ and the Lie algebra $E$ with the inverse given by
\begin{align*}
\phi^{-1}(y):= (p(y), y - p(y)),  \text{ \ for all \ } y \in E.
\end{align*}
As $\phi$ is a linear isomorphism, there exists a unique Rota-Baxter Lie algebra structure on $\lie \times V$ with  the Rota-Baxter operator and the product on $\lie \times V$  given by
\begin{align}
\np((g,x)):=\phi^{-1}(\pe(\phi(g,x))), \quad  [(g, x), (h, y)] := \phi^{-1}([\phi(g, x), \phi(h, y)]),
\label{eq:newone}
\end{align}
for all $g, h \in \lie$ and $x, y \in V$,
such that $\phi$ is an isomorphism of Rota-Baxter Lie algebras. In fact, by the proof of Theorem~3.4 in~\cite{AM14}, the product in Eq.~(\ref{eq:newone}) coincides with the one defined by Eq.~(\ref{eq:uniproduct}). Hence, it remains to prove that the Rota-Baxter operator in Eq.~(\ref{eq:newone}) coincides with the one defined by Eq.~(\ref{eq:unilie}). Indeed, for any $g \in \lie$ and $x \in V$, we have
\begin{align*}
&\ \np((g,x)) =\phi^{-1}(\pe(\phi(g,x)))=\phi^{-1}(\pe(g+x))\\
=&\ \phi^{-1}(\pe(g)+\pe(x))=\phi^{-1}(\plie(g)+\pe(x)) \\
&\ \quad \quad \quad \quad \text{(by $\rblie$ being a Rota-Baxter Lie subalgebra of $(E,\pe)$)}\\
=&\ \phi^{-1}(\plie(g)+ p(\pe(x))+ \pe(x)-p(\pe(x)))\\
=&\ (\plie(g)+p(\pe(x)), \pe(x)-p(\pe(x))) \quad \text{(by $p \circ p=p$ and $p \circ (id_E-p)=0$)}\\
=&\ (\plie(x)+P_1(x), P_2(x)).
\end{align*}
Finally, we can check the following diagram is commutative
\begin{align*}
\xymatrix{
  \rblie \ar[d]_{id_{\lie}} \ar[r]^{i_{\lie}} & (\gv, \np) \ar[d]_{\phi} \ar[r]^{q} & V \ar[d]^{id_V} \\
  \rblie \ar[r]^{i_{\lie}} & (E,P_E) \ar[r]^{\pi} & V   }
\end{align*}
where $\pi: E \rightarrow V$ is the projection of $E =\lie + V$ on the vector space $V$ and $q : \gv \rightarrow V, q(g,x):=x$ is the canonical projection. Hence $\phi$ stabilizes $\rblie$ and $V$. This completes the proof.
\end{proof}

\begin{remark}
By Theorem~\ref{thm:indu}, any Rota-Baxter Lie algebra structure on a vector space $E$ containing $\rblie$ as a Rota-Baxter Lie subalgebra is isomorphic to a unified product of $\rblie$ through a given complement $V$ of $\lie$ in $E$. Hence the classification of all Rota-Baxter Lie algebra structures on $E$ that contains $\rblie$ as a Rota-Baxter Lie subalgebras is equivalent to the classification of all unified products $\gv$, associated to all Rota-Baxter Lie extending structures $\Omega(\rblie, V)=(\tril, \trir, f, \{-,- \}, P_1, P_2)$, for a given complement $V$ of $\lie$ in $E$.
\end{remark}

To classify all unified products $\gv$, we need the following lemma.

\begin{lemma}
Let $\Omega(\rblie, V)=(\tril, \trir, f, \{-,- \}, P_1, P_2)$, $\Omega'(\rblie, V)=(\tril', \trir', f', \{-,-\}', P'_1, P'_2)$ be two Rota-Baxter Lie extending structures of $\rblie$ through $V$ and let $\gv, \lie \natural' V$ be their corresponding  associated unified products. Then there exists a bijection between the set of all morphisms of Rota-Baxter Lie algebras $\psi: \gv \rightarrow \lie \natural' V$ which stabilizes $\rblie$, and the set of pairs $(r,v)$, where $r: V \rightarrow \lie$, $v: V \rightarrow V$ are two linear maps satisfying the following conditions: for any $g \in \lie$, $x, y \in V$,
\begin{enumerate}
\item \label{it:a1} $v(x) \tril' g=v(x \tril g)$;

\item \label{it:a2} $r(x \tril g)=[r(x), g]-x \trir g+v(x) \trir' g$;

\item \label{it:a3} $v(\{x,y \})=\{ v(x), v(y)\}'+v(x) \tril' r(y)-v(y) \tril' r(x)$;

\item \label{it:a4} $r(\{ x,y \})=[r(x), r(y)]+ v(x) \trir' r(y)-v(y) \trir' r(x)+f'(v(x), v(y))-f(x,y)$;

\item \label{it:a5} $P_1(x)=\plie(r(x))+P'_1(v(x))-r(P_2(x))$;

\item \label{it:a6} $v(P_2(x))=P'_2(v(x))$.
\end{enumerate}
Under the above bijection, the morphism of Rota-Baxter Lie algebras $\psi=\psi_{(r,v)}: \gv \rightarrow \lie \natural' V$ corresponding to $(r,v)$ is given by
\begin{align*}
\psi(g,x)=(g+r(x), v(x)),
\end{align*}
for any $g \in \lie$ and $x \in V$. Moreover, $\psi= \psi_{(r,v)}$ is an isomorphism of Rota-Baxter Lie algebras if and only if $v : V \rightarrow V$ is a bijection and $\psi= \psi_{(r,v)}$ co-stabilizes $V$ if and only if $v= id_V$.
\label{lem:iso}
\end{lemma}

\begin{proof}
By the proof of Lemma~3.5 in~\cite{AM14}, a linear map $\psi: \gv \rightarrow \lie \natural' V$ which stabilizes $\lie$ is uniquely determined by two linear maps $r: V \rightarrow \lie$ and $v: V \rightarrow V$ such that $\psi(g,x)=(g+r(x), v(x))$ for all $g\in \lie$ and $x \in V$. Moreover, $\psi$ is a morphism of Lie algebras if and only if the conditions~\ref{it:a1}-\ref{it:a4} hold. Denote by $\np$ and $\np'$ the Rota-Baxter operators associated to $\gv$ and $\lie \natural' V$. Next, we show that $\psi$ is compatible with the Rota-Baxter operators, i.e. $\psi \circ \np=\np' \circ \psi$, if and only if the conditions~\ref{it:a5}-\ref{it:a6} hold. For $g, h \in \lie$, we have
\begin{align*}
\psi \circ \np ((g,0))=\psi((\plie(g),0))=(\plie(g),0)= \np'((g,0))=\np' \circ \psi((g,0)).
\end{align*}
$\psi$ is compatible with the Rota-Baxter operators under this case. For $x \in V$, we have
\begin{align*}
&\ \psi \circ \np((0,x))-\np' \circ \psi((0,x))=\psi((P_1(x), P_2(x)))-\np'((r(x),v(x)))\\
=&\ (P_1(x)+r(P_2(x)), v(P_2(x)))-(\plie(r(x))+P'_1(v(x)), P'_2(v(x))).
\end{align*}
Hence $\psi \circ \np((0,x))=\np' \circ \psi((0,x))$ if and only if the conditions~\ref{it:a5}-\ref{it:a6} hold.

Moreover, by the  proof of Lemma~3.5 in~\cite{AM14}, $\psi= \psi_{(r,v)}$ is an isomorphism if and only if $v : V \rightarrow V$ is a bijection and $\psi= \psi_{(r,v)}$ co-stabilizes $V$ if and only if $v= id_V$.
\end{proof}

By Lemma~\ref{lem:iso}, we give the following definition.
\begin{defn}\label{def:equiv}
Let $\rblie$ be a Rota-Baxter Lie algebra and let $V$ be a vector space. Two Rota-Baxter Lie extending structures of $\rblie$ through $V$, $\Omega(\rblie, V)=(\tril, \trir, f, \{-,-\}, P_1, P_2)$ and $\Omega'(\rblie, V)=(\tril', \trir', f', \{-.-\}', P'_1, P'_2)$ are called {\bf equivalent},  denoted by $$\Omega(\rblie, V) \equiv \Omega'(\rblie, V),$$ if there exists a pair $(r,v)$ of linear maps, where $r: V \rightarrow \lie$ and $v \in Aut_{\bfk}(V)$ such that $$(\tril', \trir', f',\{-,-\}', P'_1, P'_2)$$ is obtained from $(\tril, \trir, f, \{-,-\}, P_1, P_2)$ using $(r,v)$ via
\begin{align*}
x \tril' g=&\ v(v^{-1}(x) \tril g),\\
x \trir' g=&\ r(v^{-1}(x) \tril g)+ v^{-1}(x) \trir g+[g, r(v^{-1}(x))],\\
f'(x,y)=&\ f(v^{-1}(x), v^{-1}(y))+r(\{ v^{-1}(x), v^{-1}(y) \})+[r(v^{-1}(x)), r(v^{-1}(y))]-r(v^{-1}(x) \tril r(v^{-1}(y)))\\
&\ -v^{-1}(x) \trir r(v^{-1}(y))+r(v^{-1}(y)\tril r(v^{-1}(x)))+ v^{-1}(y) \trir r(v^{-1}(x)),\\
\{x,y\}'=&\ v(\{v^{-1}(x),v^{-1}(y) \})-v(v^{-1}(x) \tril r(v^{-1}(y)))+v(v^{-1}(y) \tril r(v^{-1}(x))),\\
P'_1(x)=&\ P_1(v^{-1}(x))-\plie(r(v^{-1}(x)))+r(P_2(v^{-1}(x))),\\
P'_2(x)=&\ v(P_2(v^{-1}(x))).
\end{align*}
\end{defn}

Recall that ${Extd}(E,\rblie)$ is the set of all equivalence classes of Rota-Baxter Lie algebra structures on $E$ which stabilizes $\rblie$. Now we give a classification of ${Extd}(E,\rblie)$ as our second main results.

\begin{theorem}
Let $\rblie$ be a Rota-Baxter Lie algebra, $E$ a vector space that contains $\lie$ as a subspace and $V$ a complement of $\lie$ in $E$. Then:
\begin{enumerate}
\item \label{it:aa} $\equiv$ is an equivalence relation on the set $\mathcal{L}(\rblie, V)$ of all Rota-Baxter Lie extending structures of $\rblie$ through $V$. Denote by $H^2_{\rblie}(V, \rblie) := \mathcal{L}(\rblie, V)/ \equiv$, the quotient set.

\item \label{it:ab} The map
\begin{align*}
\Phi: H^2_{\rblie}(V, \rblie) \rightarrow Extd(E,\rblie), \quad \overline{(\tril,\trir, f,\{-,-\},P_1,P_2)} \mapsto (\lie \natural  V, [-, -],\np)
\end{align*}
is a bijection, where $\overline{(\tril,\trir, f,\{-, -\},P_1,P_2)}$ is the equivalence class of $(\tril,\trir, f,\{-,-\},P_1,P_2)$
via $\equiv$.
\end{enumerate}
\label{thm:main}
\end{theorem}

\begin{proof}
\ref{it:aa} By Theorems~\ref{thm:cha},~\ref{thm:indu} and Lemma~\ref{lem:iso}, $\Omega(\rblie, V)\equiv \Omega'(\rblie, V)$ in the sense of Definition~\ref{def:equiv} if and only if there exists an isomorphism of Rota-Baxter Lie algebras $\psi: \gv \rightarrow \lie \natural'  V$ which stabilizes $\rblie$. Therefore, $\equiv$ is an equivalence relation on the set $\mathcal{L}(\rblie, V)$ of all Rota-Baxter Lie extending structures.

\ref{it:ab} By Lemma~\ref{lem:iso}, the map $\Phi$ is a well-defined injection. By Theorem~\ref{thm:indu}, $\Phi$ is a surjection. Hence $\Phi$ is a bijection.
\end{proof}

\begin{remark}
Let $\Omega(\rblie, V)=(\tril, \trir, f, \{-,-\}, P_1, P_2)$ and $\Omega'(\rblie, V)=(\tril', \trir', f', \{-.-\}', P'_1,\\ P'_2)$ be
two Rota-Baxter Lie extending structures of $\rblie$ through $V$. They are called {\bf cohomologous}, denoted by $\Omega(\rblie, V) \approx \Omega'(\rblie, V)$, if $\tril'=\tril$, $P'_2=P_2$ and there exists a linear map $r: V \rightarrow \lie$ such that
\begin{align*}
x \trir' g &\ =x \trir g+r(x \tril g)-[r(x),g],\\
f'(x,y) &\ =f(x,y)+r(\{x,y\})+[r(x),r(y)]+ y \trir r(x)-x \trir r(y)+r(y \tril r(x))-r(x \tril r(y)),\\
\{ x,y\}' &\ =\{x,y\}-x \tril r(y)+y \tril r(x),\\
P'_1(x) &\ =P_1(x)-P_{\lie}(r(x))+r(P_2(x)),
\end{align*}
for all $g \in \lie$ and $x,y \in V$.

Similar to the proof of Theorem~\ref{thm:main}, we can conclude that $\Omega(\rblie, V) \approx \Omega'(\rblie, V)$ if and only if there exists an isomorphism of Rota-Baxter Lie alebras $\phi: \gv \rightarrow \lie \natural'  V$ which stabilizes $\rblie$ and co-stabilizes $V$. Then $\approx$ is also an equivalent relation on the set $\mathcal{L}(\rblie, V)$ of all Rota-Baxter Lie extending structures of $\rblie$ through $V$. Define
$$H^2(V, \rblie) :=\mathcal{L}(\rblie, V) / \approx,$$
then the map
\begin{align*}
H^2(V,\rblie) \rightarrow Extd'(E, \lie), \quad \overline{(\tril, \trir, f, \{-,-\})} \mapsto (\gv,[-,-], \np)
\end{align*}
is a bijection between $H^2(V,\rblie)$ and the isomorphism classes of all Rota-Baxter Lie algebras structures on $E$ which stabilizes $\rblie$ and co-stabilizes $V$.
\end{remark}

\section{Special cases of  unified products}\label{sec:b}
In the section, we first introduce the notion of crossed products and bicrossed products of Rota-Baxter Lie algebras. Then we prove that both products are two special cases of the unified products of Rota-Baxter Lie algebras.

{\em Throughout the rest of this paper}, we shall omit the map if it is one of the maps $\tril, \trir, f, \{-,-\},P_1$ or $P_2$ of an extending datum $\Omega(\rblie, V)=(\tril, \trir, f, \{-,- \}, P_1, P_2)$ and it is trivial.
\subsection{Crossed products and the extension problem}
In this subsection, we introduce the concept of a crossed product of Rota-Baxter Lie algebras and show that the extension problem studied by Mishra-Das-Hazra~\cite{MDH23} is a special case of the Rota-Baxter Lie extending structures problem.

Let $\Omega(\rblie, V)=(\tril, \trir, f, \{-,-\})$ be an extending datum of $\rblie$ through $V$ such that $\tril$ is a trivial map, i.e. $x \tril g=0$, for all $x \in V$ and $g \in \lie$. It follows from  Theorem~\ref{thm:cha} that $\Omega(\rblie, V)=(\trir,f, \{-,-\})$ is a Rota-Baxter Lie extending structure of $\rblie$ through $V$ if and only if $(V, \{-,-\}, P_2)$ is a Lie algebra that satisfy the following relations:
\begin{align*}
&\ f(x,x)= 0,\\
&\ x \trir [g,h]= [x \trir g,h]+[g, x \trir h],\\
&\ \{x,y\} \trir g= x \trir(y \trir g)-y \trir(x \trir g)+[g,f(x,y)],\\
&\ f(x,\{y,z\}+f(y,\{z,x\})+f(z,\{x,y\})+x \trir f(y,z)+y \trir f(z,x)+z \trir f(x,y)=0,\\
&\ [P_\lie(g),P_1(y)]-P_2(y) \trir P_\lie(g)+P_\lie(y \trir P_\lie(g))-P_\lie([g,P_1(y)])+P_\lie(P_2(y) \trir g)\\
&\ + \lambda P_\lie(y \trir g)=0,\\
&\ [P_1(x), P_1(y)]+P_2(x) \trir P_1(y)-P_2(y) \trir P_1(x)+f(P_2(x),P_2(y))+P_\lie(y \trir P_1(x))\\
&\ - P_\lie (f(P_2(x),y))
-P_1(\{P_2(x),y \})-P_\lie (x \trir P_1(y))-P_\lie( f(x, P_2(y)))-P_1(\{x,P_2(y)\})\\
&\ - \lambda P_\lie (f(x,y))- \lambda P_1 (\{ x,y\})=0.
\end{align*}
for any $g,h \in \lie$ and $x,y,z \in V$.

In this case, the associated unified product $\gv$ is the {\bf crossed product} of the Rota-Baxter Lie algebras $\rblie$ and $(V,P_2)$. A system $(\rblie, (V,P_2), \trir, f, P_1)$ consisting of two Rota-Baxter Lie algebras $\rblie, (V,P_2)$, two bilinear maps $\trir: V \times \lie \rightarrow \lie, f: V \times V \rightarrow \lie$ and a linear map $P_1: V \rightarrow \lie$ satisfying the above conditions is called a {\bf crossed system of Rota-Baxter Lie algebras}. The Rota-Baxter Lie algebra $(\gv, [-,-], \np)$ associated to the crossed system $(\rblie, (V,P_2), \trir, f, P_1)$ is defined as follows: for any $g,h \in \lie$ and $x,y \in V$,
\begin{align*}
[(g,x), (h,y)] :=([g,h]+ x\trir h-y \trir g+f(x,y), \{ x,y\}), \quad \np((g,x)):=(P_\lie(g)+P_1(x), P_2(x)).
\end{align*}
Note that the crossed system provides the answer to the following restricted version of the extending structures problem:
\begin{prob}
Let $\rblie$ be a Rota-Baxter Lie algebra, $E$ a vector space containing $\lie$ as a subspace. Describe and classify all Rota-Baxter Lie algebra structures on $E$ such that $\rblie$ is a Rota-Baxter ideal of $E$.
\end{prob}
Let $(\rblie, (V,P_2), \trir, f,P_1)$ be a crossed system of two Rota-Baxter Lie algebras. Since
\begin{align*}
[(g,0), (h,y)]=([g,h]- y \trir g,0)\  \text{ and }\ \np((g,0))=(P_\lie(g),0),
\end{align*}
we have that $\lie \cong \lie \times \{ 0\}$ is a Rota-Baxter ideal of the Rota-Baxter Lie algebra $\gv$. Next we will prove that any Rota-Baxter Lie algebra structure on a vector space $E$ containing $\rblie$ as a Rota-Baxter ideal can be described by a crossed product of Rota-Baxter Lie algebras.

\begin{coro}
Let $\rblie$ be a Rota-Baxter Lie algebra and let $E$ be a vector space containing $\lie$ as a subspace. Then any Rota-Baxter Lie algebra structure on $E$ that contains $\rblie$ as a Rota-Baxter ideal is isomorphic to a Rota-Baxter Lie algebra associated to a crossed system $(\rblie, (V,P_2), \trir, f,P_1)$.
\label{cor:crossed}
\end{coro}

\begin{proof}
Suppose $(E,\pe)$ is a Rota-Baxter Lie algebra such that $\rblie$ is a Rota-Baxter ideal in $(E,\pe)$. Then $\rblie$ is also a Rota-Baxter Lie subalgebra of $(E,\pe)$. Hence by Theorem~\ref{thm:indu}, there is a Rota-Baxter Lie extending structure $\Omega(\rblie, V)=(\tril,\trir,f, \{-,-\})$ of $\rblie$ through a subspace $V$ of $E$ and an isomorphism of Rota-Baxter Lie algebras $E \cong \gv$. As $\rblie$ is a Rota-Baxter ideal of $(E,\pe)$, for any $x \in V$ and $g \in \lie$, we have that $[x,g] \in \lie$ and hence $p([x,g])=[x,g]$, where $p$ is the linear map in the proof of Theorem~\ref{thm:indu}. Thus, $x \tril_p g=[x,g]-p([x,g])=0$  and $(E,\pe)$ is isomorphic to the Rota-Baxter Lie algebra associated to the crossed system $(\rblie, (V,P_2), \trir, f,P_1)$, where $V=\mathrm{ker}(p)$.
\end{proof}

\begin{defn}\cite[Definition~3.1]{MDH23}
Let $\rblie$ and $(\mathfrak{h},P_{\mathfrak{h}})$ be two Rota-Baxter Lie algebras. A {\bf non-abelian extension} of $\rblie$ by $(\mathfrak{h},P_{\mathfrak{h}})$ is a Rota-Baxter Lie algebra $(\mathfrak{k}, P_{\mathfrak{k}})$ equipped with a short exact sequence of Rota-Baxter Lie algebras
\begin{align*}
0 \rightarrow \mathfrak{h} \rightarrow \mathfrak{k} \rightarrow \mathfrak{g} \rightarrow 0.
\end{align*}
\end{defn}

The equivalent classes of non-abelian extension of $\rblie$ by $(\mathfrak{h},P_{\mathfrak{h}})$ is characterized in~\cite{MDH23}. Actually, the restricted version of the extending structures problem is in face an equivalent reformulation of the non-abelian extension problem. Indeed,  the Rota-Baxter Lie algebra $\gv$ associated to a crossed system $(\rblie, (V,P_2), \trir, f,P_1)$ is an extension of $(V,P_2)$ by $\rblie$ via the following sequence
\begin{align*}
0 \rightarrow \lie \stackrel{i}{\longrightarrow} \gv \stackrel{\pi}{\longrightarrow} V \rightarrow 0,
\end{align*}
where $i: \lie \rightarrow \gv$, $i(g)=(g,0)$ and $\pi:\gv \rightarrow V$, $\pi(g,x)=x$. Conversely, let $(\mathfrak{k}, P_{\mathfrak{k}})$ be an extension of $(V,P_2)$ by $\rblie$. Then there is an exact sequence of Rota-Baxter Lie algebras
\begin{align*}
0 \rightarrow \lie \stackrel{i}{\longrightarrow} \mathfrak{k} \stackrel{\pi}{\rightarrow} V \rightarrow 0.
\end{align*}
As $\lie \cong Im(i)=Ker(pi)$, $\rblie$ can be viewed as a Rota-Baxter ideal of $(\mathfrak{k}, P_{\mathfrak{k}})$. By Corollary~\ref{cor:crossed}, there exists a crossed system $(\rblie, (V,P_2), \trir, f)$ of Rota-Baxter Lie algebras such that $\gv \cong \mathfrak{k}$ as Rota-Baxter Lie algebras.

\subsection{Matched pairs and the factorization problem}

In this subsection, we first recall the concept of a matched pair of Lie algebras~\cite{LW90, Maj90} and then show that how the factorization problem studied by Lang and Sheng~\cite{LS23} is a special case of Theorem~\ref{thm:main}.
\begin{defn}
{\bf A matched pair of Lie algebras} consists of two Lie algebras $(\lie,[-,-])$ and $(V,\{-,-\})$, which satisfy that $\lie$ is a left $V$-module under $\trir: V \otimes \lie \rightarrow \lie$, $V$ is a right $\lie$-module under $\tril: V \otimes \lie \rightarrow V$ and for all $g,h \in \lie$ and $u,v \in V$,
\begin{align*}
u \trir [g,h]=&\ [u \trir g,h]+[g, u \trir h]+(u \tril g) \trir h-(u \tril h) \trir g,\\
\{u,v \} \tril g=&\ \{u, v \tril g \}+ \{u \tril g, v \}+ u \tril (v \trir g)- v \tril (u \trir g).
\end{align*}
\end{defn}

The notion of representations of Rota-Baxter associative algebras was introduced in~\cite{GL19} and further studied in~\cite{QP18}. The concept of representations of Rota-Baxter Lie algebras of weight 0 was given in~\cite{JS21} and the general notion of representations of Rota-Baxter Lie algebras was given in~\cite{LS23} in the study of matched pairs of Rota-Baxter Lie algebras.

\begin{defn}\cite[Definition~3.12]{LS23}
A {\bf matched pair of Rota-Baxter Lie algebras} consists of two Rota-Baxter Lie algebras $\rblie$ and $(\mathfrak{h}, P_{\mathfrak{h}})$, which satisfy that $\rblie$ is a left $(\mathfrak{h}, P_{\mathfrak{h}})$-module under $\trir: \mathfrak{h} \otimes \lie \rightarrow \lie$, $(\mathfrak{h}, P_{\mathfrak{h}})$ is a right $\rblie$-module under $\tril: V \otimes \lie \rightarrow V$ and $(\lie, \mathfrak{h}, \tril, \trir)$ is a matched pair of Lie algebras.
\end{defn}

Let $\Omega(\rblie, V)=(\tril,\trir, f,\{-,-\}, P_1,P_2)$ be an extending datum of $\rblie$ through $V$ such that $f$ and $P_1$ are the trivial maps, i.e. $f(x,y)=0$ and $P_1(x)=0$, for all $x,y \in V$. Then by Theorem~\ref{thm:cha}, $\Omega(\rblie, V)=(\tril,\trir, \{-,-\}, P_2)$ is a Rota-Baxter Lie extending structure of $\rblie$ through $V$ if and only $(V, \{-,-\}, P_2)$ is a Rota-Baxter Lie algebra and $(\rblie, (V,P_2), \tril, \trir)$ is a matched pair of Rota-Baxter Lie algebras. In this case, we denote the associated unified product $\gv$ by $ \lie \bowtie V$ and call it the {\bf bicrossed product} of the matched pair $(\rblie, (V,P_2), \tril, \trir)$ of Rota-Baxter Lie algebras.

The bicrossed product of two Rota-Baxter Lie algebras is the construction which provides the answer for the following factorization problem.

\begin{prob}
Let $\rblie$ and $(\mathfrak{h}, P_{\mathfrak{h}})$ be two Rota-Baxter Lie algebras. Describe and classify all Rota-Baxter Lie algebras $(\mathfrak{k}, P_{\mathfrak{k}})$ that factorize through $\rblie$ and $(\mathfrak{h}, P_{\mathfrak{h}})$, i.e. $(\mathfrak{k}, P_{\mathfrak{k}})$  contains $\rblie$ and $(\mathfrak{h}, P_{\mathfrak{h}})$ as Rota-Baxter Lie subalgebras such that $\mathfrak{k}=\lie+ \mathfrak{h}$ and $\lie \cap \mathfrak{h}=\{0\}$.
\end{prob}
By Proposition~3.14 in~\cite{LS23}, a Rota-Baxter Lie algebra $(\mathfrak{k}, P_{\mathfrak{k}})$ factorizes through $\rblie$ and $(\mathfrak{h}, P_{\mathfrak{h}})$ if and only if there is a matched pair of Rota-Baxter Lie algebras $(\rblie, (\mathfrak{h}, P_{\mathfrak{h}}), \tril,\trir)$ such that $\lie \bowtie \mathfrak{h} \cong (\mathfrak{k}, P_{\mathfrak{k}})$. Hence the factorization problem can be restated as:

\begin{prob}\label{pro:factor}
Let $\rblie$ and $(\mathfrak{h}, P_{\mathfrak{h}})$ be two Rota-Baxter Lie algebras. Describe the set of all matched pairs $(\rblie, (\mathfrak{h}, P_{\mathfrak{h}}), \tril,\trir)$ and classify up to an isomorphism all bicrossed product  $\lie \bowtie \mathfrak{h}$.
\end{prob}

As a bicrossed product of a matched pair $(\rblie, (V,P_2), \tril, \trir)$ of Rota-Baxter Lie algebras is a special case of the unified product of an extending datum $\Omega(\rblie, V)=(\tril,\trir, f,\{-,-\}, P_1,P_2)$ of $\rblie$ through $V$ such that $f$ and $P_1$ are the trivial maps, Problem~\ref{pro:factor} can be characterized by Theorem~\ref{thm:main} as a special case.

\section{Flag extending structures}\label{sec:c}
In this section, we apply our main theorem to a flag extending structure of Rota-Baxter Lie algebras as a special case.

Given a vector space $E$ containing a Rota-Baxter Lie algebra $\rblie$ as a subspace, there is a complement $V$ of $\lie$ in $E$. Although Theorem~\ref{thm:main} gives an classification of the extending structure problem of Rota-Baxter Lie algebras, there is a computational problem.
\begin{prob}
How to compute the object $H^2_{\rblie}(V, \rblie)$ explicitly and how to obtain all Rota-Baxter Lie algebra structures on $E$ which contains $\rblie$ as a Rota-Baxter Lie subalgebra?
\end{prob}

Now we provide a way of answering the problem for a special case.

\begin{defn}
Let $\rblie$ be a Rota-Baxter Lie algebra and let $E$ be a vector space containing $\lie$ as a subspace. A Rota-Baxter Lie algebra on $E$ such that $\rblie$ is a Rota-Baxter Lie subalgebra is called a {\bf flag extending structure} of $\rblie$ to $E$ if there is a finite chain of Rota-Baxter Lie subalgebras of $E$
\begin{align*}
\rblie=(E_0, P_0) \subset (E_1, P_1) \subset \cdots \subset (E_m,P_m)=(E,\pe)
\end{align*}
such that $E_i$ has codimension 1 in $E_{i+1}$, for all $i=0, \ldots, m-1.$
\end{defn}

All flag extending structures of $\rblie$ to $E$ can be described recursively: the first step is to describe and classify all unified products $\lie \natural V_1$, for a 1-dimensional vector space $V_1$; then replace $\rblie$ with $\lie \natural V_1$, describe and classify all unified products of $(\lie \natural V_1) \natural V_2$, where $V_1$ and $V_2$ are vector spaces of dimension 1. After $m$ steps, we obtain all flag extending structure of $\rblie$ to $E$.

Now we show that all unified products of $\lie \natural V$ with $V$ a 1-dimensional vector space are parameterized by the space $ExDer(\rblie)$ of all extended derivations of $\rblie$.
Now we recall the notion of twisted derivation of a Lie algebra~\cite{AM14}.
\begin{defn}\cite[Definition~5.2]{AM14}
Let $\lie$ be a Lie algebra. A {\bf twisted derivation} of $\lie$ is a pair $(\varepsilon,D)$, where $\varepsilon: \lie \rightarrow \bfk$  and $D: \lie \rightarrow \lie$ are two linear maps such that for all $g, h \in \lie$,
\begin{align}
\varepsilon([g,h])=&\ 0, \label{eq:d1}\\
D([g,h])=&\ [D(g), h]+[g, D(h)]+ \varepsilon(g) D(h)-\varepsilon(h) D(g) \label{eq:d2}.
\end{align}
\end{defn}

\begin{defn}
Let $\rblie$ be a Rota-Baxter Lie algebra. A {\bf extended derivation} of $\rblie$ is a quadruple $(\varepsilon,D,g_0, k_0)$, where $\varepsilon: \lie \rightarrow \bfk$ and $D: \lie \rightarrow \lie$ are two linear maps, $g_0 \in \lie$ and $k_0 \in \bfk$ such that $(\varepsilon,D)$ is a twisted derivation of $\lie$ and for all $g \in \lie$,
\begin{align}
[\plie(g), g_0]&\ -k_0 D(\plie(g)) +\plie (D (\plie(g)))+ \varepsilon(\plie(g)) g_0-\plie([g,g_0])\nonumber\\
&\ +\plie(k_0 D(g))
+k_0 \varepsilon(g)g_0
+ \lambda \plie(D(g))+\lambda \varepsilon(g) g_0= 0, \label{eq:d3}\\
&\ k_0^2 \varepsilon(g)+\lambda k_0 \varepsilon(g)= 0 \label{eq:d4}.
\end{align}
\end{defn}
Denote the set of all extended derivations of $\rblie$ by $ExDer(\rblie)$. We shall now prove that the set of all Rota-Baxter Lie extending structures $\mathcal{L}(\rblie, V)$ of a Rota-Baxter Lie algebra $\rblie$ through a 1-dimensional vector space $V$ is parameterized by $ExDer(\rblie)$.

\begin{prop}
Let $\rblie$ be a Rota-Baxter Lie algebra and let $V$ be a vector space of dimension 1 with a basis $\{x\}$. Then there exists a bijection between the set $\mathcal{L}(\rblie, V)$ of all Rota-Baxter Lie extending structures of $\rblie$ through $V$ and the space $ExDer(\rblie)$ of all extended derivations of $\rblie$. Through the above bijection, the Rota-Baxter Lie extending structure $\Omega(\rblie, V)=(\tril, \trir, f, \{-,-\}, P_1, P_2)$ associated to $(\varepsilon,D, g_0,k_0) \in ExDer(\rblie)$ is given by
\begin{align*}
x \tril g=\varepsilon(g)x,\, x \trir g=D(g),\, f=0, \, \{-,-\}=0, P_1(x)=g_0, \, P_2(x)=k_0 x,
\end{align*}
for all $g \in \lie$ and $x \in V$. Denote the corresponding unified product $\gv$ by $\lie \natural_{(\varepsilon,D,g_0,k_0)} V$.
\end{prop}

\begin{proof}
Note that the set $\mathcal{L}(\rblie, V)$ of all Rota-Baxter Lie extending structures of $\rblie$ through $V$ is equivalent to the set of all the maps $\{(\tril, \trir, f, \{-,-\}, P_1, P_2) \}$, where $\tril: V \times \lie \rightarrow \lie, \trir: V \times \lie \rightarrow V, f: V \times V \rightarrow \lie, \{-,-\}: V \times V \rightarrow V, P_1:V \rightarrow \lie$ and $P_2: V \rightarrow V$ satisfying the conditions~\ref{it:1}-\ref{it:11} of Theorem~\ref{thm:cha}.
Since $V$ has dimension 1, the condition~\ref{it:1} is equivalent to that $f$ and $\{-,-\}$ are trivial maps. Again by the dimension of $V$ is 1, there is a bijection between the set of all bilinear maps $\tril: V \times \lie \rightarrow V$ and the set of all linear maps $\varepsilon: \lie \rightarrow \bfk$. The bijection is given such that the action $\tril: V \times \lie \rightarrow V$ associated to $\varepsilon$ is given by the formula: $x \tril g=\varepsilon(g)x$, for all $g \in \lie$. Similarly, any bilinear map $\trir: V \times \lie \rightarrow \lie$ is uniquely determined by a linear map $D: \lie \rightarrow \lie$ via the formula: $x \trir g=D(g)$, for all $g \in \lie$. Moreover, the linear map $P_1: V \rightarrow \lie$ is uniquely determined by an element $g_0 \in \lie$ via the formula: $P_1(x)=g_0$ and the linear map $P_2: V \rightarrow V$ is uniquely determined by an element $k_0 \in \bfk$ via the formula: $P_2(x)=k_0 x$.

Since $dim_{\bfk}(V)=1,f=0, \{-,-\}=0$, we find that the conditions~\ref{it:4}-\ref{it:7},~\ref{it:10}-\ref{it:11} hold automatically.
Moreover, the condition~\ref{it:2} is equivalent to Eq.~(\ref{eq:d1}) holds, the condition~\ref{it:3} is equivalent to Eq.~(\ref{eq:d2}) holds, the condition~\ref{it:8} is equivalent to Eq.~(\ref{eq:d3}) holds and the condition~\ref{it:9} is equivalent to Eq.~(\ref{eq:d4}) holds. This completes the proof.
\end{proof}

\begin{defn}
Let $(\varepsilon,D, g_0,k_0)$ and $(\varepsilon',D',g'_0,k'_0)$ be two extended derivations of a Rota-Baxter Lie algebra $\rblie$. $(\varepsilon,D, g_0,k_0)$ and $(\varepsilon',D',g'_0,k'_0)$ are {\bf equivalent}, denoted by $(\varepsilon,D, g_0,k_0) \equiv (\varepsilon',D',g'_0,k'_0)$, if $\varepsilon=\varepsilon', k_0=k'_0$ and there exists an element $g_1 \in \lie$ and a nonzero element $k_1 \in \bfk$ such that for any $g \in g$,
\begin{align}
D(g)=&\ k_1 D'(g)+[g_1, g]-\varepsilon(g)g_1, \label{eq:e1}\\
g_0=&\ \plie(g_1)+k_1 g'_0-k_0 g_1. \label{eq:e2}
\end{align}
\end{defn}

\begin{theorem}
Let $\rblie$ be a Rota-Baxter Lie algebra of codimension 1 in a vector space $E$ and let $V$ be a complement of $\lie$ in $E$ with basis $\{x\}$. Then $\equiv$ is an equivalent relation of the set $ExDer(\rblie)$ of all extended derivations of $\rblie$ and $Extd(E, \lie) \cong H^2_{\rblie}(V, \rblie) \cong ExDer(\rblie)/ \equiv$.
\end{theorem}

\begin{proof}
By Lemma~\ref{lem:iso}, there is an isomorphism of Rota-Baxter Lie algebras between $\lie \natural_{(\varepsilon,D, g_0,k_0)} V$ and $\lie \natural_{(\varepsilon',D', g'_0,k'_0)}$ which stabilizes $\rblie$ if and only if there are linear maps $r: V \rightarrow \lie$ and $v: V \rightarrow V$ satisfying the conditions~\ref{it:a1}-\ref{it:a6} and $v$ is a bijection. Since $V=\bfk \{x\}$, the map $r: V \rightarrow \lie$ is uniquely determined by an element $g_1 \in \lie$ such that $r(x)=g_1$ and the bijective map $v: V \rightarrow V$ is uniquely determined by a nonzero element $k_1 \in \bfk$ such that $v(x)=k_1 x$. Note that $f=f'=0$ and $\{-,-\}=\{-,-\}'=0$ in the corresponding Rota-Baxter Lie extending structures, the conditions~\ref{it:a3} and~\ref{it:a4} hold trivially. The condition~\ref{it:a1} is equivalent to $k_1 \varepsilon'(g)x=k_1 \varepsilon(g)x$, for all $g \in \lie$. As $k_1 \neq 0$, $\varepsilon=\varepsilon'$. The condition~\ref{it:a2} is equivalent to Eq.~(\ref{eq:e1}) and the condition~\ref{it:a5} is equivalent to Eq.~(\ref{eq:e2}). Moreover, the condition~\ref{it:a6} is equivalent to $k_0k_1x=k'_0k_1x$ and hence $k_0=k'_0$ by $k_1 \neq 0$. Thus $\lie \natural_{(\varepsilon,D,g_0,k_0)} V \cong \lie \natural_{(\varepsilon',D',g'_0,k'_0)} V$ as Rota-Baxter Lie algebras if and only if $(\varepsilon,D,g_0,k_0) \equiv (\varepsilon',D',g'_0,k'_0)$. This completes the proof.
\end{proof}

\smallskip

\noindent {\bf Acknowledgments}: This research is supported by the National Natural Science Foundation of China (Grant No.\@12101316).


\end{document}